\documentclass[twoside]{article}
\usepackage{aistats2018_arxiv}
\usepackage[utf8]{inputenc} 
\usepackage[T1]{fontenc}    
\usepackage{hyperref}       
\usepackage{url}            
\usepackage{booktabs}       
\usepackage{nicefrac}       
\usepackage{microtype}      

\usepackage[round]{natbib}
\usepackage{graphicx}

\usepackage{amsmath}
\usepackage{amsfonts}
\usepackage{amssymb}
\usepackage{amsthm}
\usepackage{color}

\theoremstyle{theorem}
\newtheorem{theorem}{Theorem}
\newtheorem{lemma}{Lemma}

\newtheorem{proposition}{Proposition}

\theoremstyle{definition}
\newtheorem{definition}{Definition}

\newtheorem{assumption}{Assumption}

\newcommand{\err}[1]{\text{err}(#1)}

\newcommand{\X}{\mathcal{X}}

\newcommand{\Xspl}{\mathbf{X}}
\newcommand{\Yspl}{\mathbf{Y}}
\newcommand{\V}{\mathcal{V}}

\newcommand{\real}{{\rm I\!R}}

\newcommand{\expec}{\mathbb{E} \,}

\newcommand{\norm}[1]{\left\|#1\right\|}

\newcommand{\abs}[1]{\left|#1\right|}
\newcommand{\paren}[1]{\left(#1\right)}
\newcommand{\braces}[1]{\left\{#1\right\}}
\newcommand{\pr}[1]{\mathbb{P}\left(#1\right)}
\newcommand{\prf}[2]{\mathbb{P}_{#1}\left(#2\right)}

\newcommand{\ind}[1]{\mathbf{1}\left\{#1\right\}}
\DeclareMathOperator*{\argmax}{argmax}

\newcommand{\Ep}{\mathbb{E}}
\newcommand{\Prob}{\mathbb{P}}

\title{Achieving the time of $1$-NN, but the accuracy of $k$-NN}

\author{
  Lirong Xue \\
  Princeton University, ORFE \\
  \texttt{lirongx@princeton.edu} \\
  \And
  Samory Kpotufe \\
  Princeton University, ORFE \\
  \texttt{samory@princeton.edu}
}

\begin{document}

%

%

\twocolumn[

\aistatstitle{Achieving the time of $1$-NN, but the accuracy of $k$-NN}

\aistatsauthor{ Lirong Xue \And Samory Kpotufe  }

\aistatsaddress{ Princeton University \And Princeton University } ]

\begin{abstract}
 We propose a simple approach which, given distributed computing resources, can nearly achieve the accuracy of $k$-NN prediction, while matching (or improving) the faster prediction time of $1$-NN. The approach consists of aggregating \emph{denoised} $1$-NN predictors over \emph{a small number} of distributed subsamples. We show, both theoretically and experimentally, that small subsample sizes suffice to attain similar performance as $k$-NN, without sacrificing the computational efficiency of $1$-NN. 
\end{abstract}

	\section{INTRODUCTION}
While $k$-Nearest Neighbor ($k$-NN) classification or regression can achieve significantly better prediction accuracy that $1$-NN ($k = 1$), practitioners often default to $1$-NN as it can achieve much faster prediction that scales better with large sample size $n$. In fact, much of the commercial tools for nearest neighbor search remain optimized for $1$-NN rather than for $k$-NN, further biasing practice towards $1$-NN. Unfortunately, $1$-NN is statistically inconsistent, i.e., its prediction accuracy plateaus early as sample size $n$ increases, while $k$-NN keeps improving longer 
for choices of $k \xrightarrow{n \to \infty} \infty$. 

In this work we consider having access to \emph{a small number} of distributed computing units, and ask whether better tradeoffs between $k$-NN and $1$-NN can be achieved by harnessing parallelism at prediction time.   
A simple idea is \emph{bagging} multiple $1$-NN predictors computed over distributed subsamples; however this tends to require large number of subsamples, while the number of computing units is often constrained in practice. In fact, 
an \emph{infinite} number of subsamples is assumed in all known consistency guarantees for the $1$-NN bagging approach \citep{biau2010rate, samworth2012optimal}. Here, we are particularly interested in small numbers of distributed subsamples (say 1 to 10) as a practical matter. Hence, we consider a simple variant of the above idea, consisting of  aggregating a few \emph{denoised} 1-NN predictors. With this simple change, we obtain the same theoretical error-rate guarantees as for $k$-NN, using few subsamples, while individual processing times are of the same order or better than $1$-NN's computation time. 

The main intuition behind denoising is as follows. The increase in variance due to subsampling is hard to counter if too few predictors are aggregated. We show that this problem is suitably addressed by \emph{denoising} each subsample as a preprocessing step, i.e., replacing the subsample labels with $k$-NN estimates based on the original data. Prediction then consists of aggregating -- by averaging or by majority -- the $1$-NN predictions from a few denoised subsamples (of small size $m \ll n$).




Interestingly, as shown both theoretically and experimentally, we can let the subsampling ratio $(m/n)\xrightarrow{n\to \infty}0$ while achieving a prediction accuracy of the same order as that of $k$-NN. Such improved accuracy over vanilla $1$-NN is verified experimentally, even for relatively small number of distributed predictors. Note that, in practice, we aim to minimize the number of distributed predictors, or equivalently the number of computing units which is usually costly 
in its own right. This is therefore a main focus in our experiments. In particular, we will see that even with a single denoised $1$-NN predictor, i.e., one computer, we can observe a significant improvement in accuracy over vanilla $1$-NN while maintaining the prediction speed of $1$-NN. Our main focus in this work is classification -- perhaps the most common form of NN prediction -- but our results readily extend to regression. 

\subsection*{Detailed Results And Related Work}
While nearest neighbor prediction methods are among the oldest and most enduring in data analysis \citep{fix1951discriminatory, cover1967nearest, kulkarni1995rates}, their theoretical performance in practical settings is still being elucidated. For statistical consistency, it is well known that one needs a number $k\xrightarrow{n\to \infty}\infty$ of neighbors, i.e., the vanilla $1$-NN method is inconsistent for either regression or classification~\citep{devroye1994strong}. In the case of regression, \cite{SK:78} shows that convergence rates ($l_2$ excess error over Bayes) behave as $O(n^{-2/(2+d)})$, for Lipschitz regression functions over data with intrinsic dimension $d$; this then implies a rate of $O(n^{-1/(2+d)})$ for binary classification via known relations between regression and classification rates (see e.g. \cite{DGL:73}). Similar rates are recovered in \citep{cannings2017local} under much refined parametrization of the marginal input distribution, while a recent paper of \cite{moscovich2016minimax} recovers similar rates in semisupervised settings. 

Such classification rates can be sharpened by taking into account the \emph{noise margin}, i.e., the mass of data away from the decision boundary. This is done in the recent work of \cite{chaudhuri2014rates} which 
obtain faster rates of the form $O(n^{-\alpha(\beta +1)/(2\alpha + d)})$ -- where the regression function is assumed $\alpha$-smooth -- which can be much faster for large $\beta$ (characterizing the noise margin). However such rates require large number of neighbors $k = O(n^{2\alpha/(2\alpha + d)})$ growing as a root of sample size $n$; such large 
$k$ implies much slower prediction time in practice, which is exacerbated by the scarcity of optimized tools for `$k$' nearest neighbor search. In contrast, fast commercial tools for $1$-NN search are readily available, building on various space partitioning data structures \citep{KL:76, C:74, BKL:71, gionis1999similarity}. 

In this work we show that the classification error of the proposed approach, namely aggregated denoised $1$-NN's, is 
of the same optimal order $\tilde O(n^{-\alpha(\beta +1)/(2\alpha + d)})$ plus a term $\tilde O(m^{-\alpha(\beta + 1)/d})$ where $m \leq n$ is the subsample size used for each denoised $1$-NN. This additional term, due to subsampling, is of lower order provided $m = \tilde \Omega (n^{d/(2\alpha + d)})$; in other words we can let the subsampling ratio
$(m/n) = \tilde \Omega (n^{-2\alpha/(2\alpha + d)}) \xrightarrow{n\to \infty} 0$ while achieving the same rate as $k$-NN. 
We emphasize that the smaller the subsampling ratio, the faster the prediction time: rather than just maintaining the prediction time of vanilla $1$-NN, we can actually get considerably better prediction time using smaller subsamples, while at the same time considerably improving prediction accuracy towards that of $k$-NN. Finally notice that the theoretical subsampling ratio of $\tilde \Omega (n^{-2\alpha/(2\alpha + d)})$ is best with smaller $d$, the \emph{intrinsic} dimension of the data, which is not assumed to be known a priori. Such intrinsic dimension $d$ is smallest for structured data in $\real^D$, e.g. data on an unknown manifold, or sparse data, and therefore suggests that much smaller subsamples -- hence faster prediction times -- are possible with structured data while achieving good prediction accuracy.  

As mentioned earlier, the even simpler approach of \emph{bagging} $1$-NN predictors is known to be consistent \citep{biau2010layered, biau2010rate, samworth2012optimal}, however only in the case of an infinite bag size, corresponding to an infinite number of computing units in our setting -- we assume one subsample per computing unit so as to maintain or beat the prediction time of $1$-NN. Interestingly, as first shown in \citep{biau2010layered, biau2010rate}, the subsampling ratio $(m/n)$ can also tend to $0$ as $n\to \infty$, while achieving optimal prediction rates (for fixed $\alpha = 1$, $\beta = 0$), albeit assuming an infinite number of subsamples. 
In contrast we show optimal rates -- on par with those of $k$-NN -- for even one \emph{denoised} subsample. 
This suggests, as verified experimentally, that few such denoised subsamples are required for good prediction accuracy. 

The recent work of \cite{kontorovich2015bayes}, of a more theoretical nature, considers a similar question as ours, and derives a \emph{penalized} $1$-NN approach shown to be statistically consistent unlike vanilla $1$-NN. The approach of \cite{kontorovich2015bayes} roughly consists of finding a subsample of the data whose induced $1$-NN achieves a significant margin between classes (two classes in that work). Unfortunately finding such subsample can be prohibitive 
 (computable in time $O(n^{4.376})$) in the large data regimes of interest here. In contrast, our training phase only involves random subsamples, and cross-validation over a denoising parameter $k$ (i.e., our training time is akin to the usual $k$-NN training time). 

Finally, unlike in the above cited works, our rates are established for multiclass classification (for the sake of completion), and depend logarithmically on the number of classes. Furthermore, as stated earlier, our results extend beyond classification to regression, and in fact are established by first obtaining \emph{regression} rates for estimating the so-called regression function $\expec[Y| X]$. 

{\bf Paper outline.}
Section \ref{sec:prelim} presents our theoretical setup and the prediction approach. Theoretical results are discussed in Section \ref{sec:overview}, and the analysis in Section \ref{sec:analysis}. Experimental evaluations on real-world datasets are presented in Section \ref{sec:experiments}. 


\section{PRELIMINARIES}
\label{sec:prelim} 
\subsection{Distributional Assumptions}
Our main focus is classification, although our results extend to regression. 
Henceforth we assume we are given an i.i.d. sample 
$\paren{\Xspl, \Yspl} \doteq \paren{X_i, Y_i}_1^n$ where $X\in \X \subset \real^D$, and 
$Y \in  [L] \doteq \braces{1, 2, \ldots, L}$.

The conditional distribution $P_{Y|X}$ is fully captured by the so-called \emph{regression function}, defined as 
$\eta: \X \mapsto [0, 1]^L$, where $\eta_i(x) = \pr{Y = i | X = x}$.  We assume the following on $P_{X, Y}$.

\begin{assumption}[Intrinsic dimension and regularity of $P_X$]
\label{Assp_int_dim}
First, for any $x\in \X$ and $r>0$, define the \emph{ball} $B(x, r)\doteq \braces{x'\in \X: \norm{x - x'} \leq r}$.  
We assume, there exists an integer $d$, and a constant $C_d$ such that, for all $x \in \X, r>0$, 
we have $P_X(B(x, r)) \geq C_d r^d$. 
\end{assumption}

In this work $d$ is unknown to the procedure.  However, as is now understood from previous work (see e.g. \cite{SK:78}), the performance of NN methods depends on such intrinsic $d$. We will see that the performance of the approach of interest here would also depends on such unknown $d$. 
In particular, as is argued in \citep{SK:78}, $d$ is low for low-dimensional manifolds, or sparse data, so we would think of $d\ll D$ for structured data. Note that the above assumption also imposes regularity on $P_X$, namely by ensuring sufficient mass locally on $\X$ (so that NNs of a point $x$ are not arbitrarily far from it). 

\begin{assumption}[Smoothness of $\eta$]
\label{Assp_smooth}
The function $\eta$ is $(\lambda, \alpha)$-H\"older for some $\lambda>0, 1\geq \alpha>0$, i.e., 
$$ \forall x, x'\in \X, \quad \norm{\eta(x) - \eta(x')}_{\infty} \leq \lambda \norm{x - x'}^\alpha.$$
\end{assumption}

We will use the following version of Tsybakov's noise condition \citep{audibert2007fast}, adapted to the multiclass setting. 

\begin{assumption}[Tsybakov noise condition]
\label{Assp_noise_margin} 
For any $x\in \X$, let $\eta_{(l)}(x)$ denote the $l$'th largest element in $\{\eta_l(x)\}_{l=1}^L$. 
There exists $\beta>0$, and $C_\beta>0$ such that 
$$ \forall t> 0, \quad \pr { \abs{\eta_{(1)}(X) - \eta_{(2)}(X)} \leq t}\leq C_\beta t^\beta.$$
\end{assumption}

\subsection{Classification Procedure}
For any classifier $h: \X \mapsto [L]$, we are interested in the $0$-$1$ classification error 
$$\err{h} = \prf{X, Y}{h(X) \neq Y}.$$  

It is well known that the above error is mimimized by the \emph{Bayes} classifier 
$h^*(x)  \doteq \argmax_l \eta_l(x).$ Therefore, for any estimated classifier $\hat h$, we are interested in 
the \emph{excess error} $\err{\hat h} - \err{h^*}$. We first recall the following basic nearest neighbor estimators. 

\begin{definition}[$k$-NN prediction]
\label{def:kNN}
Given $k \in \mathbb{N}$, let $k\text{NN-I}(x)$ denote the indices of the $k$ nearest neighbors of $x$ in the sample 
$\Xspl$. Assume, for simplicity, that ties are resolved so that $\abs{k\text{NN-I}(x)} = k$. 
The $k$-NN classifier can be defined via the \emph{regression} estimate $\hat \eta: \X \mapsto [0, 1]^L$, where 
\begin{align*}
  \hat \eta_l(x) \doteq \frac{1}{k} \sum_{i \in k\text{NN-I}(x)}\ind{Y_i = l}.   
\end{align*}
 The $k$-NN classifier is then obtained as:\\
  $$h_{\hat \eta}(x) = \argmax_l \hat \eta_l(x).$$

Finally we let $r_k(x)$ denote the distance from $x$ to its $k$-th nearest neighbor. 
\end{definition}

We can now formally describe the approach considered in this work. 

\begin{definition}[Denoised $1$-NN] 
  Consider a random subsample ({\bf without replacement}) $\Xspl'$ of $\Xspl$ of size $m \leq n$. For any $x\in \X$, let $\text{NN}(\Xspl'; x)$ denote the nearest neighbor of $x$ in $\Xspl'$. The \underline{denoised $1$-NN} estimate at $x$ is given as $\hat h(x) = h_{\hat \eta}(\text{NN}(\Xspl'; x))$, where $h_{\hat \eta}$ is as defined above for some fixed $k$. 

  This estimator corresponds to $1$-NN over a sample $\Xspl'$ where each $X_i' \in \Xspl'$ is {\bf prelabeled} as 
  $h_{\hat \eta}(X_i')$. 
\end{definition}

The resulting estimator, which we denote \emph{subNN} for simplicity, is defined as follows. 
\begin{definition}[subNN]
\label{def:subNN}
Let $\{\hat h_i\}_{i = 1}^I$, denote denoised $1$-NN estimators defined over $I$ {\bf independent} subsamples of size $m$ (i.e., the $I$ sets of indices corresponding to each subsample are picked independently, although the indices in each set is picked with replacement in $[n]$). At any $x\in \X$, the \underline{subNN estimate} $\bar h(x)$ is the majority label in $\{\hat h_i(x)\}_{i = 1}^I$. 
\end{definition}

It is clear that the subNN estimate can be computed in parallel over $I$ machines, while the final step -- namely, computation of the majority vote -- takes negligible time. Thus, we will view the {\bf prediction time complexity} at any query $x$ as the average time (over $I$ machines) it takes to compute the $1$-NN of $x$ on each subsample. This time complexity gets better as $(m/n)\to 0$. Furthermore, we will show that, even with relatively small $I$ (increasing variability), we can let $(m/n)$ get small while attaining an excess error on par with that of $k$-NN (here $h_{\hat \eta}$). This is verified experimentally.

\section{OVERVIEW OF RESULTS} 
\label{sec:overview}
Our main theoretical result, Theorem \ref{theo:main} below, concerns the statistical performance of subNN. 
The main technicality involves characterizing the effect of subsampling and denoising on performance. 
Interestingly, the rate below does not depend on the number $I$ of subsamples: this is due to the \emph{averaging} 
effect of taking majority vote accross the $I$ submodels, and is discussed in detail in Section \ref{sec:analysis} 
(see proof and discussion of Lemma \ref{lemma:error_h_bar}). In particular, the rate is bounded in terms of a \emph{bad event} that is unlikely for a random submodel, and therefore unlikely to happen for a majority. 

\begin{theorem}\label{theo:main}
	Let $0< \delta<1$. Let $\V$ denote the VC dimension of balls on $\X$. 
	With probability at least $1- L\delta$, there exists a choice of $k\in [n]$, such that the estimate $\bar h$ satisfies 
	\begin{align*} 
	\err{\bar h} -\err{h^*} 
	&\leq C_1 \cdot \paren{\frac{\V\ln(Ln/\delta)}{n}}^{\alpha(\beta + 1)/(2\alpha + d)} \\
	&+ C_2\cdot \paren{\frac{\V \ln(m/\delta)}{m}}^{\alpha(\beta+1)/d}, 
	\end{align*} 
	for constants $C_1, C_2$ depending on $P_{X, Y}$.  
\end{theorem}

The first term above is a function of the size $n$ of the original sample, and recovers the recent optimal bounds for $k$-NN classification of \cite{chaudhuri2014rates}. We note however that the result of \cite{chaudhuri2014rates} concerns binary classification, while here we consider the more general setting of multiclass. Matching lower bounds were established earlier in \citep{audibert2007fast}. 

The second term, a function of the subsample size $m$, characterizes the additional error (over vanilla $k$-NN $h_{\hat \eta}$) due to subsampling and due to using $1$-NN's at prediction time. As discussed earlier in the introduction, the first term dominates (i.e. we recover the same rates as for $k$-NN) whenever the subsampling ratio 
$(m/n) = \tilde \Omega(n^{-\alpha/(2\alpha + d)})$ which goes to $0$ as $n\to \infty$. This is remarkable in that it suggests smaller subsample sizes are sufficient (for good accuracy) in the large sample regimes motivating the present work. We will see later that this is supported by experiments.  

As mentioned earlier, similar vanishing subsampling ratios were shown for \emph{bagged} $1$-NN in \citep{biau2010layered, biau2010rate, samworth2012optimal}, but assuming a infinite number of subsamples. 
In contrast the above result holds for any number of subsamples, and the improvements over $1$-NN are supported in 
experiments over varying number of subsamples, along with varying subsampling ratios. 

The main technicalities and insights in establishing Theorem \ref{theo:main} are discussed in Section \ref{sec:analysis} below, with some proof details relegated to the appendix. 

\section{ANALYSIS OVERVIEW}
\label{sec:analysis}
The proof of Theorem \ref{theo:main} combines the statements of Propositions \ref{prop:agg_clas} and \ref{proposition:denoised_NN_reg} below. The main technicality involved is in establishing 
Proposition \ref{prop:agg_clas} which brings together the effect of noise margin $\beta$, smoothness $\alpha$, and 
the overall error due to denoising over a subsample. We overview these supporting results in the next subsection, followed 
by the proof of Theorem \ref{theo:main}. 

\subsection{Supporting Results}
Theorem \ref{theo:main} relies on first establishing a rate of convergence for the $k$-NN regression estimate $\hat \eta$, used in denoising the subsamples. While such rates exist 
in the literature under various assumptions (see e.g. \cite{SK:78}), we require a \emph{high-probability} rate that holds \emph{uniformly} over all $x\in \X$.  This is given in Proposition \ref{prop:kNN_reg} below, and is established for our particular setting where $Y$ takes discrete multiclass values (i.e. $\hat \eta$ and $\eta$ are both multivariate functions).  
Its proof follows standard techniques adapted to our particular aim, and is given in the appendix (supplementary material).

\begin{proposition}[Uniform $k$-NN regression error]\label{prop:kNN_reg}
      	Let $0<\delta<1$. Let $\hat \eta$ denote the $k$-NN regression estimate of  
 Definition \ref{def:kNN}. The following holds for a choice of $k = O\left( \left(\ln\frac{nL}{\delta}\right)^{\frac{d}{2\alpha + d} } (nC_d)^\frac{2\alpha}{2\alpha + d} \right)$. With probability at least $1-2\delta$ over $\bf (X,Y)$, we have simultaneously for all $x\in{\cal X}$:
      	$$\| \hat{\eta}(x) - \eta(x) \|_\infty \leq C \left(\frac{{\cal V} \ln (nL/\delta)}{nC_d} \right)^{\frac{\alpha}{2\alpha + d}}$$
      	where $C$ is a function of $\alpha$, $\lambda$.
\end{proposition}

The above statement is obtained, by first remarking that, under structural assumptions on $P_X$, (namely that there is sizable mass everywhere locally on $\X$), nearest neighbor distances can be uniformly bounded with high-probability. 
Such nearest neighbor distances control the bias of the $k$-NN estimator, while its variance behaves like $O(1/k)$. 

Such a uniform bound on NN distances is given in Lemma \ref{lemma:r_k} below and follows standard insights. 

\begin{lemma}[Uniform bound on NN distances $r_k$ ] \label{lemma:r_k}
As in Definition \ref{def:kNN}, let $r_k(x)$ denote the distance from $x\in \X$ to its $k$'th nearest neighbor in a 
sample $\Xspl \sim P_X^n$. Then, with probability at least $1-\delta$ over $\Xspl$, the following holds for all $k\in [n]$:
      	  $$\sup_{x \in {\cal X}} r_k(x) \leq \left(\frac{3}{C_d}\right)^{\frac{1}{d}} \cdot  \max \left(\frac{k}{n}, \frac{{\cal V}\ln 2n + \ln\frac{8}{\delta}}{n} \right)^\frac{1}{d}.$$
\end{lemma}

\subsection{SubNN Convergence}
We are ultimately interested in the particular regression estimates induced by subsampling: the denoised $1$-NN estimates $\hat h$ over a subsample can be viewed as 
$\hat h = \argmax_{l \in [L]} {\hat \eta^\sharp_l}$ for a regression estimate $\hat \eta^\sharp (x) \doteq \hat{\eta}(\text{NN}(\Xspl'; x))$, i.e., $\hat \eta$ evaluated at the nearest neighbor $\text{NN}(\Xspl'; x)$ of $x$ in $\Xspl'$. Our first step is to relate the error of $\hat \eta^\sharp$ to that of $\hat \eta$. Here again the bound on NN distances of Lemma \ref{lemma:r_k} above comes in handy since $\hat \eta^\sharp$ can be viewed as 
introducing additional bias to $\hat \eta$, a bias which is in turn controlled by the distance from a query $x$ to its NN in the subsample $\Xspl'$. By the above lemma, this distance is of order $\tilde{O}(m^{-1/d})$, introducing a bias of order 
$\tilde{O}(m^{-\alpha/d})$ given the smoothness of $\eta$. 

Thus, combining the above two results yields the following regression error on denoised estimates.

\begin{proposition}[Uniform convergence of denoised $1$-NN regression] \label{proposition:denoised_NN_reg}
      	Let $0<\delta<1$. Let $\hat \eta$ denote the $k$-NN regression estimate of  
 Definition \ref{def:kNN}. Let $\Xspl'$ denote a subsample (without replacement) of $\Xspl$. Define 
 the \emph{denoised} $1$-NN estimate $\hat \eta^\sharp (x) \doteq \hat{\eta}(\text{NN}(\Xspl'; x))$. 
The following holds for a choice of $k = O\left( (\V\ln\frac{nL}{\delta})^{\frac{d}{2\alpha + d} } (nC_d)^\frac{2\alpha}{2\alpha + d} \right)$. With probability at least $1-3\delta$ over $\bf (X,Y)$, and $\Xspl'$, we have simultaneously for all $x\in{\cal X}$:

  \begin{align*}
    \|\hat \eta^\sharp (x) - \eta(x) \|_\infty 
    \leq& C \left(\frac{\V\ln (nL/\delta)}{nC_d} \right)^{\frac{\alpha}{2\alpha + d}} \\
    &+ C \left(\frac{\V \ln(m/\delta)}{mC_d} \right)^\frac{\alpha}{d} \text{,}
  \end{align*}
    where $C$ is a function of $\alpha$, $\lambda$. 
\end{proposition}
\begin{proof}
	Define $x' = \text{NN}(\Xspl'; x)$ so that $\hat{\eta}^\sharp(x) = \hat{\eta}(x')$. We then have the two parts decomposition: 
	\begin{align}
	\|&\hat{\eta}^\sharp(x)  - \eta(x)\|_\infty 
	 =   \|\hat{\eta}(x') - \eta(x)\|_\infty \nonumber \\
	&\leq \|\hat{\eta}(x') - \eta(x')\|_\infty + \|\eta(x') - \eta(x)\|_\infty \nonumber\\
	&\leq C \left(\frac{{\cal V}\ln (nL/\delta)}{nC_d} \right)^{\frac{\alpha}{2\alpha + d}} + \|\eta(x') - \eta(x)\|_\infty, 
	\label{eq:sub1}
	\end{align}
where the last inequality follows (with probability $1-2\delta$) from Proposition \ref{prop:kNN_reg}. 
	
To bound the second term in inequality \eqref{eq:sub1}, notice that $\Xspl'$ can be viewed as $m$ i.i.d. samples from $P_X$. Therefore $\|x' - x\|$ can be bounded using Lemma \ref{lemma:r_k}. Therefore by smoothness condition on $\eta$ in Assumption \ref{Assp_smooth}, we have with probability at least $1-\delta$, simultaneously for all $x\in \X$:
	\begin{align}
	\|&\eta(x') - \eta(x)\|_\infty \leq \lambda \|x' - x\|^{\alpha} \nonumber \\
	  &\leq \lambda \left(\frac{3{\cal V}\ln 2m + 3\ln\frac{8}{\delta}}{m C_d} \right)^\frac{\alpha}{d}
	  \leq C \left(\frac{{\cal V} \ln \frac{m}{\delta}}{mC_d} \right)^\frac{\alpha}{d}.
	  \label{eq:sub2}
	\end{align}
	
	Combining (\ref{eq:sub1}) and (\ref{eq:sub2}) yields the statement.
\end{proof}

Next we consider \emph{aggregate regression error}, i.e., 
the discrepancy $(\eta_{h^*(x)}(x) - \eta_{\bar h(x)}(x))$ between the coordinates of $\eta$ given by the labels $h^*(x)$ and $\bar h(x)$. This will be bounded in terms of the error $\phi(n, m)$ attainable by the individual denoised regression estimates (as bounded in the above proposition). 

\begin{lemma} [Uniform convergence of aggregate regression] \label{lemma:error_h_bar}
Given independent subsamples $\braces{\Xspl_i'}_{i = 1}^I$ from $(\Xspl, \Yspl)$, define 
$\hat \eta^\sharp_i(x) = \hat \eta (\text{NN}(\Xspl_i'; x))$, i.e., the regression estimate $\hat \eta$ evaluated at 
the nearest neighbor $\text{NN}(\Xspl_i'; x)$ of $x$ in $\Xspl'_i$. 
Suppose there 
exists $\phi = \phi(n, m)$ such that, 
$$ \max_{i \in [I]}\, \prf{\Xspl, \Yspl, \Xspl_i'}{\exists x \in \X, \, \|\hat \eta^\sharp_i(x) - \eta(x)\|_\infty > \phi} \leq \delta, $$
for some $0< \delta < 1$. 
Then, let $\bar h$ denote the subNN estimate using subsamples $\braces{\Xspl_i'}_{i = 1}^I$.  With probability at least $1-L\delta$ over the randomness in $\bf (X,Y)$ and $\braces{\Xspl_i'}_{1}^I$, the following holds simultaneously for all $x\in{\cal X}$:
	\begin{align*}
	&\eta_{h^*(x)}(x) - \eta_{\bar h(x)}(x) \leq 2\phi. 
	\end{align*}
\end{lemma}

\emph{Remark.} Notice that in the above statement, the probability of error goes from $\delta$ to $L\delta$, but does not depend on the number $I$ of submodels. This is because of the \emph{averaging} effect of the majority vote. For intuition, suppose $B$ is a \emph{bad event} and $\mathbf{1}_{B_i}$ is whether $B$ happens for submodel $i$. Suppose further that $\expec \mathbf{1}_{B_i} \leq \delta$ for all $i$. Then the likelihood of $B$ happening for a majority of more than $I/2$ of models is 
$$\expec \ind{ {\sum_i \mathbf{1}_{B_i} \geq I/2} } \leq \frac{2}{I} \cdot {\expec \sum_i \mathbf{1}_{B_i}} \leq 2\delta,$$
by a Markov inequality. We use this type of intuition in the proof, however over a sequence of related \emph{bad} events, and using the fact that, the submodels' estimates are independent conditioned on $\Xspl, \Yspl$. 

\begin{proof}
	The result is obtained by appropriately bounding the indicator 
$\ind{ \exists x: \eta_{h^*(x)}(x) - \eta_{\bar h(x)}(x) > 2\phi }.$
	
Let $\hat h_i$ denote the denoised $1$-NN classifier on sample $\Xspl_i'$, or for short, the $i$'th submodel. First notice that, if the majority vote $\bar h(x) \doteq l$ for some label $l\in [L]$, then at least $I/L$ submodels $\hat h_i(x)$ predict $l$ at $x$. In other words, we have  
	\begin{align*}
	 \frac{L}{I} \sum_{i=1}^{I} \ind{\hat h_i(x) = l} \geq 1.
	\end{align*}
	
	Therefore, fix $x \in \cal X$, and let $\bar h(x) \doteq l$; we then have:
	\begin{align}
	\mathbf{1} \, \big \{ &\eta_{h^*(x)}(x) - \eta_{\bar h(x)}(x) > 2\phi  \big \} \nonumber\\
		&\doteq \ind{ \eta_{h^*(x)}(x) - \eta_{l}(x) > 2\phi } \nonumber\\
	    &\leq \frac{L}{I} \sum_{i=1}^{I} \ind{\hat h_i(x) = l} \ind{\eta_{h^*(x)}(x) - \eta_{l}(x) > 2\phi } \nonumber\\
		&\leq \frac{L}{I} \sum_{i=1}^{I} \ind{\eta_{h^*(x)}(x) - \eta_{\hat h_i(x)}(x) > 2\phi }. \label{line:lemma5:ind:12345} 
	\end{align}
	
	We bound the above as follows. Suppose $\hat h_i(x) \doteq l_i$ and $h^*(x) \doteq l^*$ for some labels 
	$l_i$ and $l^*$ in $[L]$. 
	Now, if $\| \eta(x) - \hat\eta^\sharp_i (x) \|_\infty \leq \phi$, then $\eta_{l^*}(x) \leq \hat\eta^\sharp_{i,l^*} (x) + \phi$ and $\eta_{l_i}(x) \geq \hat\eta^\sharp_{i,l_i} (x) - \phi$. Also by definition we know that $\hat h_i(x)$ is the maximum entry of $\hat\eta^\sharp_i (x)$, so $\hat\eta^\sharp_{i,l^*} (x) \leq \hat\eta^\sharp_{i,l_i} (x)$. Therefore
	 \begin{align*} 
	 \eta_{h^*(x)}(x) &- \eta_{\hat h_i(x)}(x) \leq (\hat\eta^\sharp_{i,l^*} (x) + \phi) - 
	 (\hat\eta^\sharp_{i,l_i} (x) - \phi)\\
	 &=  (\hat\eta^\sharp_{i,l^*} (x) - \hat\eta^\sharp_{i,l_i} (x)) + 2\phi \leq 2\phi.
	 \end{align*} 
	In other words, ${\eta_{h^*(x)}(x) - \eta_{\hat h_i(x)}(x) > 2\phi } $ only when $\| \eta(x) - \hat\eta^\sharp_i (x) \|_\infty > \phi$. Thus, bound (\ref{line:lemma5:ind:12345}) to obtain:
	\begin{align*}
	  & \ind{\eta_{h^*(x)}(x) - \eta_{\bar h(x)}(x) > 2\phi } \\
	  &\leq \frac{L}{I} \sum_{i=1}^{I} \ind{\| \eta(x) - \hat\eta^\sharp_i (x) \|_\infty > \phi }.
	\end{align*}
	
	Finally we use the fact that, for an event $A(x)$, we have $\ind{\exists x \in {\cal X}: A(x)} = \sup_{x \in {\cal X}} \ind{A(x)}$. Combine this fact with the above inequality to get:
	\begin{align*}
	& \ind{\exists x \in {\cal X}: \eta_{h^*(x)}(x) - \eta_{\bar h(x)}(x) > 2\phi } \\
	& = \sup_{x \in {\cal X}} \ind{\eta_{h^*(x)}(x) - \eta_{\bar h(x)}(x) > 2\phi } \\
	& \leq \sup_{x \in {\cal X}} \frac{L}{I} \sum_{i=1}^{I} \ind{\| \eta(x) - \hat\eta^\sharp_i (x) \|_\infty > \phi } \\
	& \leq \frac{L}{I} \sum_{i=1}^{I} \sup_{x \in {\cal X}} \ind{\| \eta(x) - \hat\eta^\sharp_i (x) \|_\infty > \phi } \\
	& = \frac{L}{I} \sum_{i=1}^{I} \ind{\exists x \in {\cal X}: \| \eta(x) - \hat\eta^\sharp_i (x) \|_\infty > \phi }.
	\end{align*}
The final statement is obtained by integrating both sides of the inequality over the randomness in $\Xspl, \Yspl$ and the (conditionally) independent subsamples $\braces{\Xspl_i'}_{1}^I$. 
\end{proof}

{Next, Proposition \ref{prop:agg_clas} below states that the excess error of the subNN estimate $\bar h$ can be bounded in terms of the aggregate \emph{regression} error $(\eta_{h^*} - \eta_{\bar h})$ considered in Lemma \ref{lemma:error_h_bar}. In particular, the proposition serves to account for the effect of the 
\emph{noise margin} parameter $\beta$ towards obtaining faster rates than those in terms of smoothness $\alpha$ only. } 
 
%
%
%
%

\begin{table*}[t]
	\begin{center}
		\small
		\caption{Datasets Used In Evaluating SubNN}
		\label{tab:data_description}
		\begin{tabular}{|c|c|c|c|c|c|}
			\hline
			Name		& \#train	& \#test	& \#dimension	& \#classes	& Description 	\\\hline
			MiniBooNE	& 120k 		& 10k		& 50			& 2			& Particle identification \citep{UCI-MiniBooNE} 	\\\hline
			TwitterBuzz	& 130k 		& 10k		& 50			& 2			& Buzz in social media \citep{UCI-TwitterBuzz}	\\\hline
			LetterBNG	& 34k 		& 10k		& 16			& 26		& English alphabet \citep{OpenML-LetterBNG} \\\hline
			NewsGroups20& 11k		& 7.5k		& 130k			& 20		& Document classification \citep{UCI-NewsGroups20} \\\hline
			YearPredMSD	& 34k		& 10k		& 90			& regression& Release year of songs	\citep{UCI-YearPredMSD}	\\\hline
			WineQuality	& 5.5k		& 1.0k		& 12			& regression& Quality of wine \citep{UCI-WineQuality} 	\\\hline
		\end{tabular}
	\end{center}
\end{table*}

\begin{proposition} \label{prop:agg_clas}
	Suppose there exists $\phi = \phi(n, m)$ such that, 
	with probability at least $1- \delta$ over the randomness in $(\Xspl, \Yspl)$ and the subsamples $\braces{\Xspl_i'}_{i = 1}^I$, we have simultaneously for all $x\in \X$, 
	$ \eta_{h^*(x)}(x) - \eta_{\bar h(x)}(x) < 2\phi$. 
	
	Then with probability at least $1-\delta$, the excess classification error of the estimate $\bar{h}$ satisfies:
	$$ \text{err}( \bar{h} ) - \text{err}(h^*) \leq C_\beta (2\phi)^{\beta+1}.$$
\end{proposition}
\begin{proof}
	
	Since, for any classifier $h$, $\eta_{h(x)}(x) \doteq \pr{Y = h(x)}$, the excess error of the sub-NN classifier $\bar{h}$ can be written as 
	$\expec_X \left[ \eta_{h^*(X)}(X) - \eta_{\bar{h}(X)}(X) \right]$. 
	
	Thus, the assumption in the proposition 
	statement  -- that for all $x\in X$, $ \eta_{h^*(x)}(x) - \eta_{\bar h(x)}(x) < 2\phi $ (with probability at least $1- \delta$) -- yields a trivial bound of $2\phi$ on the excess error. We want to refine this bound. 
	
	Let $\eta_{(l)}(x)$ be the $l$-th largest entry in the vector $\eta(x) = \braces{\eta_l}_{l = 1}^L$, and define $\Delta(x) \doteq \eta_{(1)}(x) - \eta_{(2)}(x)$. 
	
	Then at any fixed point $x$, we can refine the bound on excess error at $x$ (namely $\eta_{h^*(x)}(x) - \eta_{\bar{h}(x)}(x)$), by separately considering the following exhaustive conditions $A(x)$ and $B(x)$ on $x$:
	\begin{itemize}
		\item[$A$:] $\Delta(x) \geq 2\phi$, in which case the excess error is 0. This follows from $\eta_{(1)}(x) = \eta_{h^*(x)}(x)$ and that:
		\begin{align*}
		\eta_{h^*(x)}(x) - \eta_{\bar h(x)}(x) < 2\phi \leq \eta_{h^*(x)}(x) - \eta_{(2)}(x)
		\end{align*}
		
		In other words $\eta_{\bar h(x)}(x)$ is larger than $\eta_{(2)}$, so equals $\eta_{(1)}(x)\doteq\eta_{h*(x)}(x)$. 
		
		\item[$B$:] $\Delta(x) < 2\phi$, in which case the excess error cannot be refined at $x$. However, the total mass of such $x$'s is at most  $C_\beta (2\phi)^\beta$ by Tsybakov's noise condition (Assumption \ref{Assp_noise_margin}).
	\end{itemize}
	Combining these conditions, we have with probability at least $1-\delta$ that the excess error satisfies:
	\begin{align*}
	\text{err}&(\bar{h}) - \text{err}(h^*) 
	= \expec_X \left[\eta_{h^*(X)}(X) - \eta_{\bar{h}(X)}(X) \right] \\
	&\leq \expec_X \left[ 0\cdot \ind{A(X)}+ 2\phi \cdot \ind{B(X)} \right] \\
	&\leq 2\phi \cdot\expec [\ind{B(X)}] \leq 2\phi \cdot {\mathbb{P}}\left(\Delta(x) \leq 2\phi \right)\\
	&\leq C_\beta (2\phi)^{\beta+1}. 
	\end{align*}
\end{proof}

Combining the results of this section yield the main theorem whose proof is given next.

\subsection{Proof Of Theorem \ref{theo:main}}
Our main result follows easily from Propositions \ref{proposition:denoised_NN_reg} and \ref{prop:agg_clas}. This is given below. 
\begin{proof}  
Fix any $0< \delta < 1$. Note that the conditions of Lemma \ref{lemma:error_h_bar} are verified in Proposition \ref{proposition:denoised_NN_reg}, 
namely that, with probability $1-3\delta$, all regression errors (of submodels) are bounded by 
\begin{align} 
\phi \doteq\, C \paren{\frac{\V\ln (Ln/\delta)}{n}}^{\alpha/(2\alpha + d)} 
         +C \paren{\frac{\V \ln(m/\delta)}{m}}^{\alpha/d}, \label{eq:phi}
 \end{align}
 where $C$ is a constant depending on $\alpha$ and $\lambda$.
         
Next, the conditions of Proposition \ref{prop:agg_clas} are obtained in Lemma \ref{lemma:error_h_bar} with the same setting of $\phi$. It follows that with probability at least $1-3L\delta$, we have:
  \begin{align*}
       \err{\bar h} -\err{h^*}  &\leq C_\beta(2\phi)^{\beta+1}, 
  \end{align*}
  with $\phi$ as in \eqref{eq:phi}. Given that $g(x)=x^{\beta+1}$ is a convex function, we conclude by applying Jensen's inequality (viewing $(1/2C) \cdot \phi$ as an average of two terms). 
\end{proof}

\section{EXPERIMENTS}
\label{sec:experiments}
\paragraph{Experimental Setup.}
Data is standardized along each coordinate of $X$.  \\
- \emph{Fitting subNN.} We view the subsample size $m$ and the number $I$ of subsamples as exogenous parameters determined by the practical constraints of a given application domain. 
Namely, smaller $m$ yields faster prediction and is driven by prediction time requirement, while larger $I$ improves 
prediction error but is constrained by available computing units. However, much of our experiments concern the sensitivity of subNN to $m$ and $I$, and yield clear insights into tradeoffs on these choices. Thus, for any fixed choice of $m$ and $I$, we choose $k$ by 2-fold cross-validation. The search for $k$ is done in two stages: 
first, the best value $k'$ (minimizing validation error) is picked on a log-scale $\{2^i \}_{i=1}^{\lceil\log n\rceil}$, then a final choice for $k$ is made on a refined linear range $[\lceil{k'/2}\rceil-10:2k'+10]$. \\
- \emph{Fitting $k$-NN.} $k$ is also chosen in two stages as above.


Table \ref{tab:data_description} describes the datasets used in the experiments. We use $k$-$d$-tree for fast NN
search from Python \texttt{scikit-learn} for all datasets but NewsGroups20 for which we perform a direct search (due to high-dimensionality and sparsity). As explained earlier, our main focus is on classification, however theoretical insights from previous sections extend to regression, as substantiated in this section. The code in Python can be found at \url{https://github.com/lirongx/SubNN}.

\paragraph{Results.}
Our main experimental results are described in Table \ref{tab:result_subNN}, showing the relative errors (error of the method divided by that of vanilla $k$-NN) and relative prediction time (prediction time divided by that of $k$-NN) for versions of subNN($(m/n)$, $I$), where $m/n$ is the subsampling ratio used, and $I$ is the number of subsamples. 
For regression datasets, the error is the MSE, while for classification we use $0$-$1$ error.  
The prediction time reported for the subNN methods is the maximum time over the $I$ subsamples plus aggregation time, reflecting the effective prediction time for the distributed-computing settings motivating this work. 
The results support our theoretical insights, namely that subNN can achieve accuracy close or matching that of $k$-NN, while at the same time achieving fast prediction time, on par or better than those of $1$-NN.

- \emph{Sensitivity to $m$ and $I$.} As expected, better times are achievable with smaller subsample sizes, while better prediction accuracy is achievable as more subsamples tend to reduce variability. This is further illustrated for instance in Figure \ref{fig:SubNN_MiniBooNE}, 
where we vary the number of subsamples. Interestingly in this figure, for the MiniBoone dataset, the larger subsampling ratio $0.75$ yields the best accuracy over any number of subsamples, but the gap essentially disappears when enough subsamples are used. We thus have the following prescription in choosing $m$ and $I$: 
while small values of $I$ work generally well, large values of $I$ can only improve accuracy; on the other hand, subsampling ratios of $0.1$ yield good time-accuracy tradeoffs accross datasets.




- \emph{Benefits of denoising.} In Figure \ref{fig:bag1NN_comp}, we compare subNN with pure bagging of $1$-NN models. As suggested by theory, we 
see that the bagging approach does indeed require considerably more subsamples to significantly improve over the error 
of vanilla $1$-NN. In contrast, the accuracy of subNN quickly tends to that of $k$-NN, in particular for TwitterBuzz where $1$ or $3$ subsamples are sufficient to statistically close the gap, even for small subsampling ratio. This could be due to hidden but beneficial structural aspects of this data. In all cases, the experiments further highlights the benefits of our simple denoising step, as a variance reduction technique. This is further supported by the error-bars (std) over 5 repetitions as shown in Figure~\ref{fig:bag1NN_comp}.

\paragraph{Conclusion.} We propose a procedure with theoretical guarantees, which is easy to implement over distributed computing resources, and which achieves good time-accuracy tradeoffs for nearest neighbor methods. 

\newpage

\begin{table*}[htbp]
	\vspace{-0.2in}
	\begin{center}
		\small
		
		\caption{Ratios Of \emph{Error Rates} and \emph{Prediction Times} Over Corresponding Errors And Times Of $k$-NN.}
		\label{tab:result_subNN}
		\begin{tabular}{|c||ccc||ccc|}
			\hline
			\multicolumn{1}{|c||}{} & \multicolumn{3}{c||}{Relative Error} & \multicolumn{3}{c|}{Relative Time} \\
			\hline
			Data		& 1NN	& subNN(0.1,10)	& subNN(0.75,10)& 1NN		& subNN(0.1,10)	& subNN(0.75,10)\\\hline
			MiniBooNE	& 1.280	& 1.039			& {\bf 1.027}	& 0.609		& {\bf 0.247}	& 0.547			\\\hline
			TwitterBuzz	& 1.405	& {\bf 1.000}	& 1.005			& 0.550		& {\bf 0.185}	& 0.522 		\\\hline
			LetterBNG	& 1.127	& {\bf 1.086}	& 1.144			& 0.459		& {\bf 0.219}	& 0.396			\\\hline
			NewsGroups20& 1.122	& 1.206			& {\bf 1.002}	& 0.610		& {\bf 0.081}	& 0.668			\\\hline
			YearPredMSD	& 1.859	& {\bf 1.082}	& 1.110			& 0.847		& {\bf 0.025}	& 0.249 		\\\hline
			WineQuality	& 1.276	& {\bf 1.011}	& 1.018			& 0.989		& {\bf 0.885}	& 0.906 		\\\hline 
			
		\end{tabular}
			
	\end{center}
\end{table*}

\begin{figure*}[htbp]
	\vspace{-0.3in}
	\centering
	\includegraphics[height=2.25 in]{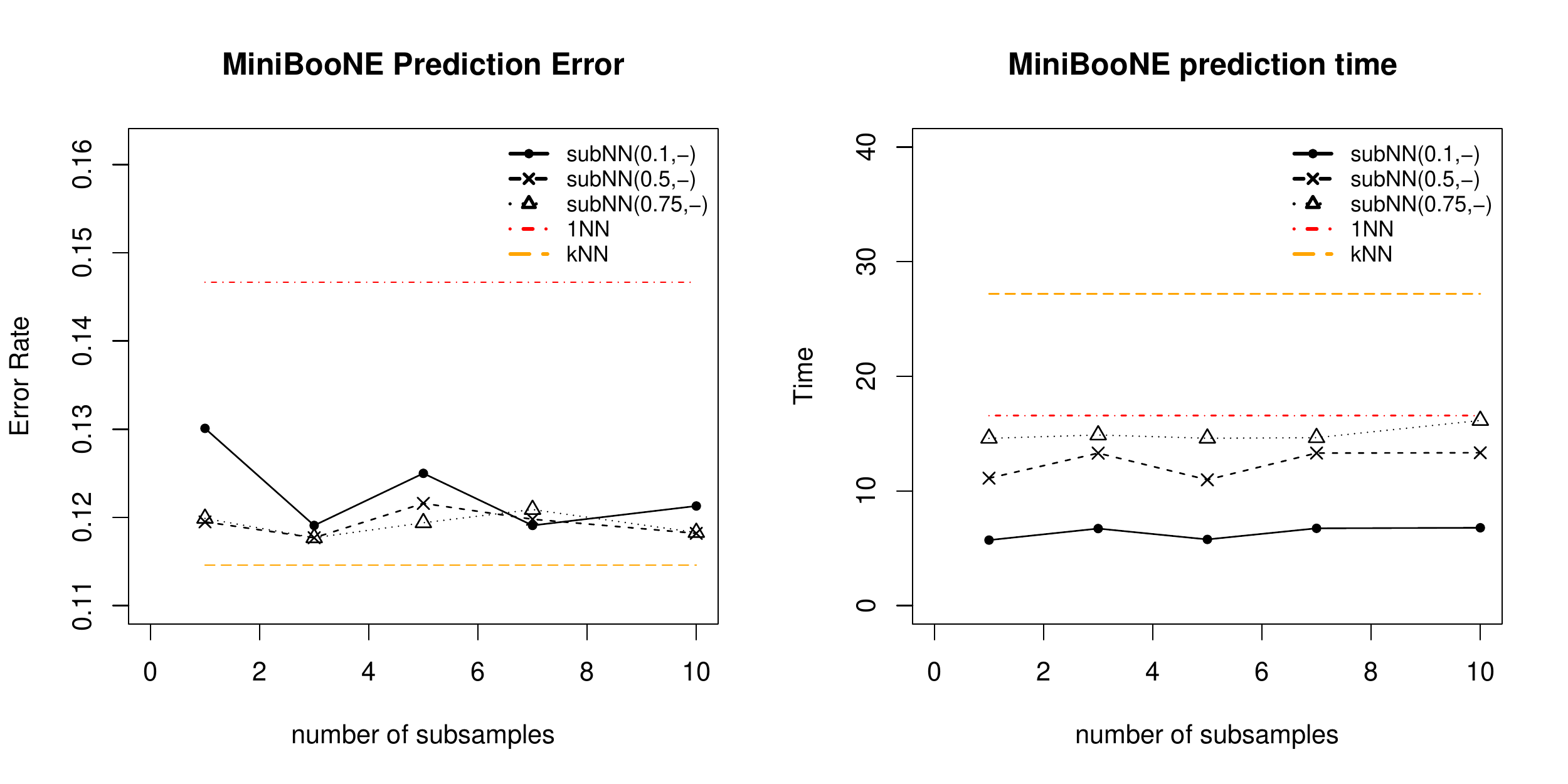}
	\vspace{-0.1in}
	\caption{Comparing the Effect Of Subsampling Ratios on Prediction And Time Performance Of SubNN. Shown are SubNN estimates using subsampling ratios 0.1, 0.5, and 0.75.}
	\label{fig:SubNN_MiniBooNE}
\end{figure*}

\begin{figure*}[htbp]
	\centering
	\includegraphics[height=2.25in]{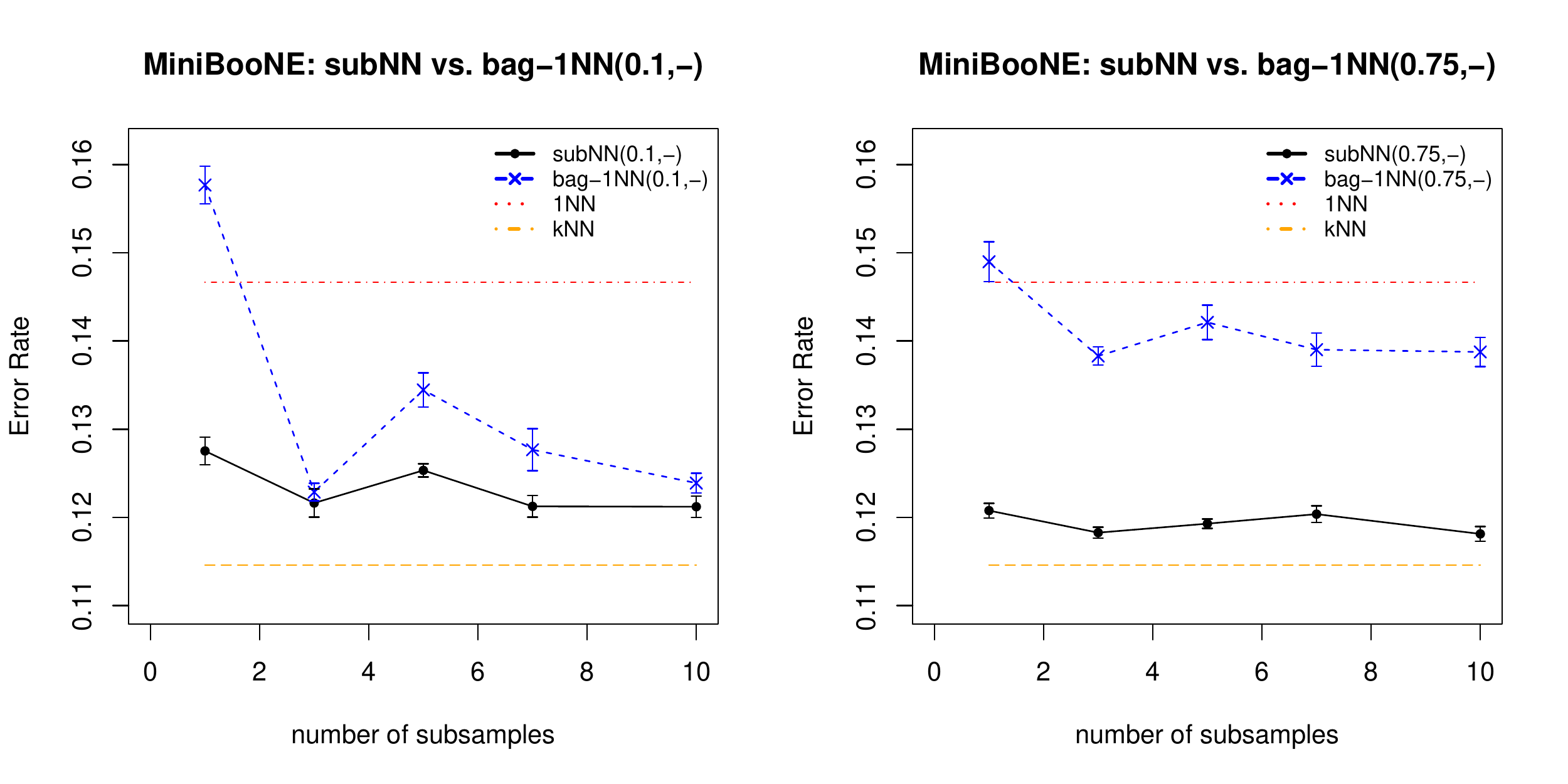} 
	\includegraphics[height=2.25in]{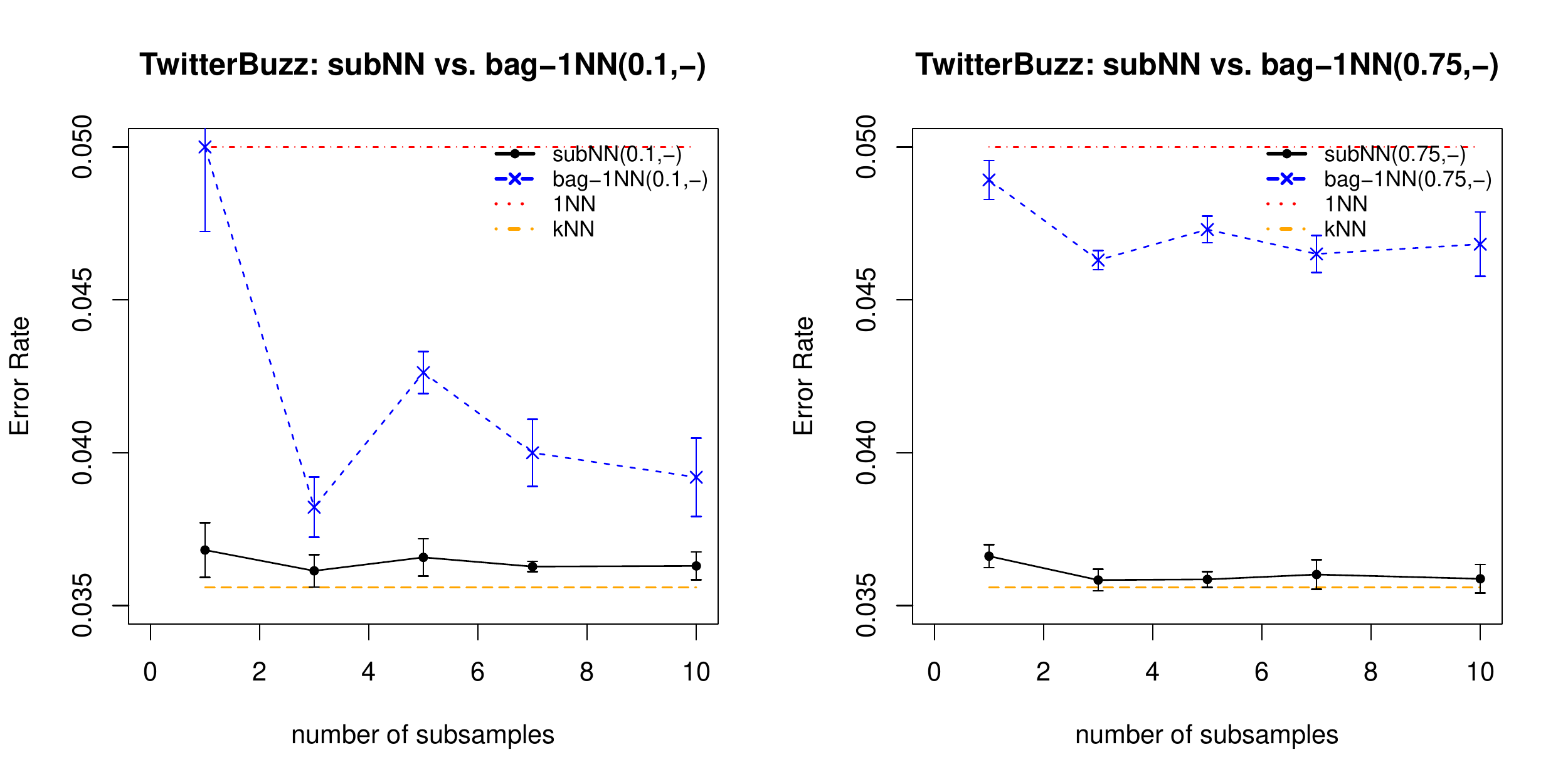} 
	\vspace{-0.1in} 
	\caption{Bagged 1NN Compared With SubNN Using Subsampling Ratios 0.1 (Left) And 0.75 (Right).}
	\label{fig:bag1NN_comp}
	\vspace{-0.25in}

\end{figure*}

\clearpage 
\begin{small}
\bibliography{refs}
\end{small}

\clearpage
 \appendix

 \section{PROOF OF MAIN RESULTS} 
  In this section, we show the proof of Proposition \ref{prop:kNN_reg} (uniform bound on $k$NN regression error). The proof is done by decomposing the regression error into bias and variance (Lemma \ref{lemma:knn_vars}) and bound each of them separately in Lemma \ref{lemma:knn_bias} and \ref{lemma:knn_vars}. Lemma \ref{lemma:r_k} will be proved first as a by-product.
  
  The decomposition is the following. Define $\tilde{\eta}(x)$ as: 
    $$\tilde{\eta}(x) = \mathbb{E}_{\bf Y|X}\hat{\eta}(x) = \frac{1}{k}\sum_{i \in k\text{NN-I}(x)} \eta(X_i) \text{.}$$
  Viewing $\tilde{\eta}$ as the expectation of $\hat{\eta}$ conditioning sample $\Xspl$, we can have the following variance and bias decomposition of error of $\hat{\eta}$:
  \begin{align}
  \| \hat{\eta}(x) - \eta(x)\|_\infty \leq \| \hat{\eta}(x) - \tilde{\eta}(x)\|_\infty + \| \tilde{\eta}(x) - \eta(x)\|_\infty \text{.}
  \label{equ:knn_var_bias_decomp}
  \end{align}
 
  We start by introducing a known result on relative VC bound in Lemma \ref{lemma:rela_VC}. Then in Lemma \ref{lemma:r_k}, we use it to give a uniform bound on $r_k(x)$ - the distance between query point $x$ and its $k$-th nearest neighborhood in data. After that we use Lemma \ref{lemma:r_k} to bound bias and variance separately in Lemma \ref{lemma:knn_bias} and \ref{lemma:knn_vars}. Proposition \ref{prop:kNN_reg} is concluded by combing the two bounds.
  
  \begin{lemma}[Relative VC Bound, Vapnik 1971] \label{lemma:rela_VC}
     Let ${\cal B}$ be a set of subsets of $\cal X$. ${\cal B}$ has finite VC dimension $\cal V$. For $n$ drawn sample $ X_1,\dots,X_n$, the empirical probability measure is defined as $\mathbb{P}_n(B) = \frac{1}{n}\sum_{i=1}^{n} {\bf 1}_{X_i \in B}$.  Define $\alpha_n = ({\cal V}\ln 2n + \ln(8/\delta))/n$. Let $0<\delta<1$, with probability at least $1-\delta$ over the randomness in $X_1,\dots,X_n$, the following holds simultaneously for all $B \in {\cal B}$:
     $$\mathbb{P}(B) \leq \mathbb{P}_n(B) + \sqrt{\mathbb{P}_n(B)\alpha_n} + \alpha_n \text{.}$$
  \end{lemma}
  


  Using the above result, we can prove Lemma \ref{lemma:r_k} (bound on $r_k$).
  \begin{proof}[{\bf Proof of Lemma \ref{lemma:r_k}}]
    By Lemma \ref{lemma:rela_VC}, let $\alpha_n = \left({\cal V}\ln 2n + \ln(8/\delta)\right)/n$, with probability at least $1 - \delta$ over the randomness in $\bf X$ and $\epsilon>0$, for any $x\in{\cal X}$ and closed ball $B_{x}(\epsilon,k) = B(x, (1-\epsilon) r_k(x))$, we have:
    \begin{align*}
      \Prob(B_{x}(\epsilon,k))
      &\leq \Prob_n(B_{x}(\epsilon,k)) + \sqrt{\Prob_n(B_{x}(\epsilon,k))\alpha_n} + \alpha_n \\
      &\leq 3\max \left(\Prob_n(B_{x}(\epsilon,k)), \alpha_n \right) \\
      &< 3\max \left(\frac{k}{n}, \frac{{\cal V}\ln 2n + \ln\frac{8}{\delta}}{n} \right) \text{.}
    \end{align*}
      	
    By Assumption \ref{Assp_int_dim} (intrinsic dimension), for any $x\in {\cal X}$ and any $\epsilon>0$, $\mathbb{P} \left( B_{x}(\epsilon,k) \right)$ is lower-bounded as below.
    
      $$\mathbb{P} \left( B_{x}(\epsilon,k) \right) =  \mathbb{P}\left( B \left( x,\left( 1-\epsilon \right) r_k \left(x \right) \right) \right) \geq C_d (1-\epsilon)^{d} r_k^d(x) \text{.}$$ 
      
    Therefore, with probability at least $1-\delta$, the following holds simultaneously for all $x\in{\cal X}$:
    \begin{align*}
      &r_k(x) \\
      &\leq \left( \frac{\mathbb{P}(B_{x}(\epsilon,k))}{C_d} \right)^{\frac{1}{d}} \\
      &\leq \left(\frac{3}{C_d}\right)^{\frac{1}{d}} \cdot \max \left(\frac{k}{n}, \frac{{\cal V}\ln 2n + \ln\frac{8}{\delta}}{n} \right)^\frac{1}{d} \cdot (1-\epsilon)^{-1} \text{.}
    \end{align*}
    
    Let $\epsilon \rightarrow 0$ in the above inequality and conclude the proof.
  \end{proof}

  Using Lemma \ref{lemma:r_k} (uniform $r_k(x)$ bound), we can get a uniform bound on the bias of $\hat{\eta}$:
  
  \begin{lemma}[Bias of $\hat{\eta}$] \label{lemma:knn_bias}
  	Let ${\cal V}$ be the VC dimension of the class of balls on ${\cal X}$. For $0<\delta<1$ and $k\geq {\cal V}\ln 2n + \ln\frac{8}{\delta}$, with probability at least $1 - \delta$ over the randomness in the choice of $\bf X$, the following inequality holds simultaneously for all $x \in {\cal X}$:
  	$$\| \tilde{\eta}(x) - \eta(x) \|_\infty \leq \lambda \cdot \left( \frac{3}{C_d} \cdot \frac{k}{n} \right)^\frac{\alpha}{d} \text{.}$$
  \end{lemma}

  \begin{proof}
  	First, for a fixed sample $\bf X$, by Assumption \ref{Assp_smooth} (Smoothness of $\eta$) we have: 
  	\begin{align*}
  	  &\left\| \tilde{\eta}(x) - \eta(x)\right\|_\infty \\
  	  &= \Big\| \frac{1}{k} \sum_{i \in k\text{NN-I}(x)}\eta(X_i) - \eta(x) \Big\|_\infty \\
  	  &\leq \frac{1}{k} \sum_{i \in k\text{NN-I}(x)} \left\| \eta(X_i) - \eta(x) \right\|_\infty \\
  	  &\leq \frac{1}{k} \sum_{i \in k\text{NN-I}(x)} \lambda \| X_i - x \|^\alpha \\
  	  &\leq \frac{1}{k}\sum_{i \in k\text{NN-I}(x)} \lambda r_k^\alpha (x) = \lambda r_k^\alpha (x) \text{.}
  	\end{align*}
  	It follows by Lemma \ref{lemma:r_k}, that, with high probability at least $1-\delta$ over $\bf X$, the following holds simultaneously for all $x \in {\cal X}$.
  	$$\| \tilde{\eta}(x) - \eta(x) \|_\infty \leq \lambda r_k^\alpha (x) \leq \lambda \cdot \left( \frac{3}{C_d} \cdot \frac{k}{n} \right)^\frac{\alpha}{d} \text{.}$$
  \end{proof}
  
  \begin{lemma}[Variance of $\hat{\eta}$] \label{lemma:knn_vars}
  	Let ${\cal V}$ be the VC dimension of balls on ${\cal X}$. For $0<\delta<1$, with probability at least $1 - \delta$ over the randomness in $\bf (X,Y)$, the following inequality holds simultaneously for all $x \in {\cal X}$:
  	
  	$$\|\hat{\eta}(x) - \tilde{\eta}(x) \|_\infty < \sqrt{\cal V} \sqrt{\frac{1}{k} \ln \frac{nL}{\delta}} \text{.} $$
  	
  \end{lemma}
  \begin{proof}

    Consider the $l$-th value of $\hat{\eta}(x) - \tilde{\eta}(x)$:
      $$ \left( \hat{\eta}(x) - \tilde{\eta}(x) \right)_l = \frac{1}{k} \sum_{i \in k\text{NN-I}(x)} \big( \ind{Y_i = l} - \eta_l(X_i) \big) \text{.} $$
  	
  	Fix $x$, $\bf X$ and consider the randomness in $\bf Y$ conditioned on $\bf X$. Use Hoeffding's Inequality. There are $k$ independent terms in the above summation and $\Ep (\ind{Y_i = l} | X_i) = \eta_l(X_i)$. So the following holds with probability at least $1 - \delta_0$ over the randomness in $\bf Y$:
  	  $$\left( \hat{\eta}(x) - \tilde{\eta}(x) \right)_l \leq \sqrt{ \frac{1}{2k} \ln \frac{2}{\delta_0} } \text{.}$$
  	Apply the above analysis to $l = 1,\dots,L$ and combine them by union bound. So the following inequality holds with probability at least $1 - L\cdot\delta_0$ over the randomness in $\bf Y$:
  	
   	  $$\|\hat{\eta}(x) - \tilde{\eta}(x) \|_\infty \leq \sqrt{ \frac{1}{2k} \ln \frac{2}{\delta_0} } \text{.}$$
  	
  	
  	Then consider variations in $x$. Given $\bf X$ fixed, the left hand side of the above inequality can be seen as a function of $k\text{NN-I}(x)$. This is a subset of $\bf X$ covered by ball $B(x, r_k(x))$. By Sauer's Lemma, the number of such subsets of $\bf X$ covered by a ball is bounded by $\big(\frac{en}{{\cal V}}\big)^{\cal V}$. So there are at most $\big(\frac{en}{{\cal V}}\big)^{\cal V}$ many different variations of the above inequality when $x$ varies. We combine the variations by union bound. Let $\delta = \delta_0 \cdot L \cdot \left(\frac{en}{{\cal V}}\right)^{\cal V}$, the following happens with probability at least $1-\delta$ over the randomness in $\bf Y$:
  	\begin{align*}
  	  \sup_{x \in {\cal X}} \| \hat{\eta}(x) - \tilde{\eta}(x) \|_\infty 
  	  &\leq \sqrt{ \frac{1}{2k} \ln \frac{2}{\delta_0} } \\
  	  &\leq \sqrt{\frac{\cal V}{2k} \ln \frac{en}{\cal V} + \frac{1}{2k}\ln \frac{2L}{\delta}} \\
  	  &\leq \sqrt{\cal V}\sqrt{\frac{1}{k} \ln \frac{nL}{\delta}}  \text{.}
  	\end{align*}
  	
  	The above inequality holds for any fixed sample $\bf X$ and the right hand side do not depend on $\bf X$. So it continue to hold for any drawn $\bf X$. Thus we conclude the proof.
  \end{proof}

      Combining the bias and variance of $\hat{\eta}$, we have the bound on uniform kNN regression error.
      
      \begin{proof}[{\bf Proof of Proposition \ref{prop:kNN_reg}}]
      	Apply Lemma \ref{lemma:knn_bias} (bias) and \ref{lemma:knn_vars} (variance) to inequality \ref{equ:knn_var_bias_decomp}, with probability at least $1-2\delta$:
      	$$\| \hat{\eta}(x) - \eta(x) \|_\infty \leq \lambda \left(\frac{3}{C_d}\right)^{\alpha/d} \cdot \left(\frac{k}{n}\right)^{\alpha / d} + \sqrt{\cal V} \cdot \sqrt{\frac{1}{k} \ln \frac{nL}{\delta}}  \text{.}$$
      	The above is minimized at $k = C' ({\cal V}\ln\frac{nL}{\delta})^{\frac{d}{2\alpha + d} } (nC_d)^\frac{2\alpha}{2\alpha + d}$, where $C'$ depend only on $\alpha$ and $\lambda$. Plug in this value of $k$ into the statement above, we obtain $\| \hat{\eta}(x) - \eta(x) \|_\infty = C'' (\frac{{\cal V} \ln (nL/\delta)}{nC_d})^{\frac{\alpha}{2\alpha + d}}$, where $C''$ depend solely on $\alpha$ and $\lambda$.
      \end{proof}
  
\section{ADDITIONAL TABLES AND PLOTS}

  In this section we present supplemental plots and tables.
  
  Table \ref{tab:result_subNN_avgTime} shows the same experiment as in Table \ref{tab:result_subNN}, but reports average prediction time (over the I subsamples) rather than the maximum prediction times, plus the aggregation time. Comparing the two tables, one can see that the differences between average and maximum times are small. In other words, prediction time is rather stable over the subsamples, which is to be expected as these times are mostly controlled by subsample size and computing resource. 

  \begin{table*}[htbp]
  	\vspace{-0.2in}
  	\begin{center}
  		\small
  		
  		\caption{Ratios Of \emph{Error Rates} and \emph{Average Prediction Times} Over Corresponding Errors And Times Of $k$-NN.}
  		\label{tab:result_subNN_avgTime}
  		\begin{tabular}{|c||ccc||ccc|}
  			\hline
  			\multicolumn{1}{|c||}{} & \multicolumn{3}{c||}{Relative Error} & \multicolumn{3}{c|}{Relative Average Time} \\
  			\hline
  			Data		& 1NN	& subNN(0.1,10)	& subNN(0.75,10)& 1NN		& subNN(0.1,10)		& subNN(0.75,10) \\\hline
  			MiniBooNE	& 1.280	& 1.039			& {\bf 1.027}	& 0.609		& {\bf 0.216}		& 0.492	\\\hline
  			TwitterBuzz	& 1.405	& {\bf 1.000}	& 1.005			& 0.550		& {\bf 0.179}		& 0.480	\\\hline
  			LetterBNG	& 1.127	& {\bf 1.086}	& 1.144			& 0.459		& {\bf 0.218}		& 0.390	\\\hline
  			NewsGroups20& 1.122	& 1.206			& {\bf 1.002}	& 0.610		& {\bf 0.076}		& 0.659 \\\hline
  			YearPredMSD	& 1.859	& {\bf 1.082}	& 1.110			& 0.847		& {\bf 0.025} 		& 0.248	\\\hline
  			WineQuality	& 1.276	& {\bf 1.011}	& 1.018			& 0.989		& {\bf 0.826}		& 0.878 \\\hline 
  			
  		\end{tabular}
  		
  	\end{center}
  \end{table*}

  Figure \ref{fig:SubNN_TwitterBuzz} presents the same experiments as in Figure~\ref{fig:SubNN_MiniBooNE}, for an additional dataset (TwitterBuzz). It compares the error and prediction time of subNN models as a function of the number $I$ of subsamples used. Again, as we can see, subNN yields error rates similar to those of kNN,  even for a small number of subsamples. As expected, the best prediction times are achieved with the smaller subsample ratio of 0.1.

  \begin{figure*}[htbp]
  	\vspace{-3.5in}
  	\centering
  	\includegraphics[height=2.5 in]{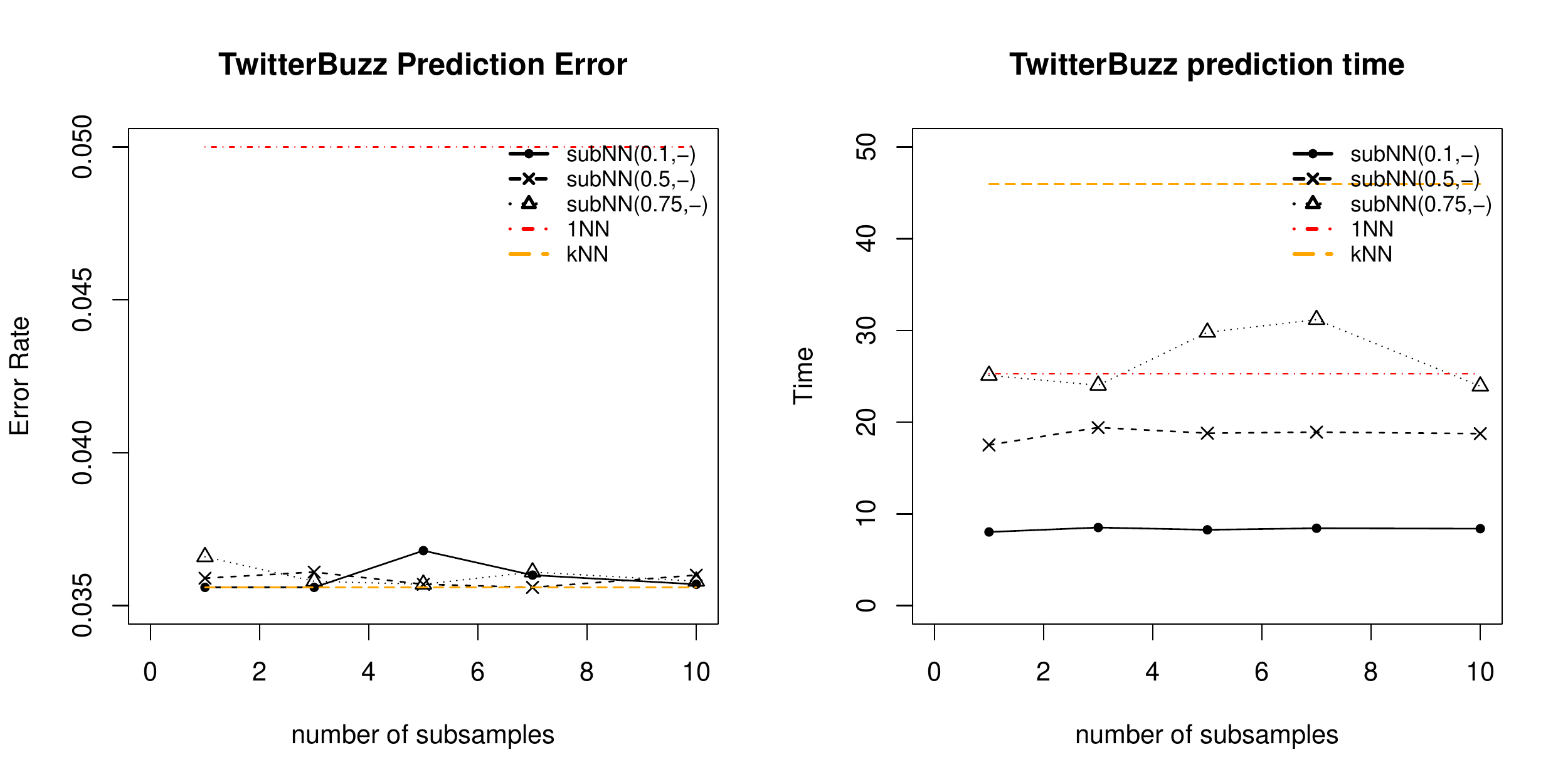}
  	\caption{Comparing the Effect Of Subsampling Ratio And Number Of Models (TwitterBuzz). We find that all the subNN predictors reach error similar to that of $k$-NN even when using only a few subsamples (1 or 3). As expected, subNN with a subsampling ratio of 0.1 results in the best prediction times. }
  	\label{fig:SubNN_TwitterBuzz}
  \end{figure*}

\end{document}